\documentclass[10pt,oneside]{amsart}
\usepackage[utf8]{inputenc}
\usepackage[english]{babel}
\usepackage{lmodern}
\usepackage[T1]{fontenc}
\usepackage{textcomp}
\usepackage{hyperref}
\usepackage{fullpage}
\usepackage{amsmath,amssymb,amsthm}
\usepackage{mathrsfs}
\usepackage[all,cmtip]{xy}
\usepackage[foot]{amsaddr}

\numberwithin{equation}{section}

\hypersetup{
pdfauthor = {Louis-Clément LEFÈVRE},
pdftitle = {Mixed Hodge structures on cohomology jump ideals},
pdfkeywords = {Mixed Hodge Structures, L-infinity Algebras, Jump Ideals, Homotopy Transfer of Structure},
}

\title{Mixed Hodge structures on cohomology jump ideals}
\author{Louis-Clément Lefèvre}
\date{\today}

\email{louisclement.lefevre@uni-due.de}
\address{Universität Duisburg-Essen, Fakultät Für Mathematik, 45117 Essen, Germany}

\theoremstyle{plain}
\newtheorem{theorem}{Theorem}[section]

\newtheorem{lemma}[theorem]{Lemma}
\newtheorem{corollary}[theorem]{Corollary}

\theoremstyle{definition}
\newtheorem{definition}[theorem]{Definition}

\theoremstyle{remark}
\newtheorem{remark}[theorem]{Remark}

\renewcommand\epsilon{\varepsilon}
\renewcommand\phi{\varphi}
\newcommand\RR{\mathbb{R}}
\newcommand\CC{\mathbb{C}}
\newcommand\ZZ{\mathbb{Z}}
\newcommand\QQ{\mathbb{Q}}

\newcommand\kk{\mathbf{k}}

\newcommand\Art{\mathbf{Art}}

\newcommand\Set{\mathbf{Set}}

\newcommand\Ohat{\widehat{\mathcal{O}}}

\newcommand\Mb{\mathbf{M}_{\mathrm{B}}}

\DeclareMathOperator\id{id}
\DeclareMathOperator\Ker{Ker}

\DeclareMathOperator*\End{End}
\DeclareMathOperator*\Hom{Hom}

\DeclareMathOperator\ad{ad}
\DeclareMathOperator\Gr{Gr}
\DeclareMathOperator\Dec{Dec}
\DeclareMathOperator\Def{Def}
\DeclareMathOperator\MC{MC}

\DeclareMathOperator\Sym{Sym}
\DeclareMathOperator\rank{rank}
\DeclareMathOperator\MHC{MHC}

\begin{document}

\begin{abstract}
In previous work, we constructed for a smooth complex variety $X$ and for a linear algebraic group $G$ a mixed Hodge structure on the complete local ring $\widehat{\mathcal{O}}_\rho$ to the moduli space of representations of the fundamental group $\pi_1(X,x)$ into $G$ at a representation $\rho$ underlying a variation of mixed Hodge structure. We now show that the jump ideals $J_k^i \subset \widehat{\mathcal{O}}_\rho$, defining the locus of representations such the the dimension of the cohomology of $X$ in degree $i$ of the associated local system is greater than $k$, are sub-mixed Hodge structures; this is in accordance with various known motivicity results for these loci. In rank one we also recover, and find new cases, where these loci are translated sub-tori of the moduli of representations. Our methods are first transcendental, relying on Hodge theory, and then combined with tools of homotopy and algebra.
\end{abstract}


\maketitle

\section{Introduction}
Let $X$ be a smooth complex variety, that can be either algebraic, or admit a compactification into a compact Kähler manifold (\emph{quasi-Kähler} varieties). Let $G$ be a linear algebraic group over $\CC$ with a fixed embedding into some $GL(N,\CC)$. Our general topic is: study the topology of $X$ via the representations into $G$ of its fundamental group $\pi_1(X,x)$.

\subsection{Jump loci}
{\hfuzz=10pt
The representations of $\pi_1(X,x)$ into $G$ are parametrized by a moduli space $\Hom(\pi_1(X,x), G)$ which also has a structure of affine scheme of finite type. Inside of this, the \emph{cohomology jump loci} are the subset defined as
\begin{equation}
\Sigma_k^i := \left\{ \rho:\pi_1(X,x)\rightarrow G(\CC)\ \big|\ \dim(H^i(X, V_\rho)) \geq k \right\}
\end{equation} }
where $V_\rho$ is the local system associated to $\rho$. These are closed sub-schemes. Given another local system $W$, one can also consider the \emph{relative loci}
\begin{equation}
\Sigma_k^i(W) := \left\{ \rho:\pi_1(X,x)\rightarrow G(\CC)\ \big|\ \dim(H^i(X, V_\rho \otimes W)) \geq k \right\}.
\end{equation}
The structure of these loci has been intensively studied. Good descriptions are known for the local structures, with applications to the topology of varieties: the main objects of interest are the completed local ring $\Ohat_\rho$ at a given representation $\rho$, and the $\Sigma_k^i(W)$ are defined locally by ideals $J_k^i(W)\subset \Ohat_\rho$. The global structure is particularly known in rank one: for $G=\CC^*$ then the moduli space of representations is the product of an algebraic torus $(\CC^*)^{b_1(X)}$ and a finite group, and in many cases the irreducible components of the jump loci are sub-tori translated by torsion points. See for example the survey by Budur-Wang \cite{BudurWangSurvey} for all these results.

More recently, various special arithmetic properties of these loci have been investigated. In \cite{BudurWangAbsolute} for example, a notion of \emph{absolute constructibility} for these loci is developed when $X$ is algebraic. This roughly means that they stay constructible after comparing with the moduli space of vector bundles with connections and applying some Galois automorphism of the field of definition of $X$. It is actually these kind of special properties that forces the sub-torus property.

These various properties indicate that the jump loci are rigid and in some sense come from geometry. Another recent result of Esnault-Kerz \cite{EsnaultKerz} states that in rank one the jump loci as sub-tori are image under the exponential of sub-mixed Hodge structures (see below) of $H^1(X,\CC)$ and they call this \emph{motivicity} result. Further motivicity properties are investigated in an article of Lerer \cite{Lerer}.

Let us mention also, without developing, that our earlier motivation was to improve a theorem of Kapovich-Millson by emphasizing the role of mixed Hodge structures: in \cite{KapovichMillson} they give some restrictions on the algebraic structure of $\Ohat_\rho$ for finite representations of smooth algebraic varieties, and use it to exhibit a class of finitely presented group that cannot be isomorphic to fundamental groups of such varieties. This method was refined by Dimca-Papadima-Suciu \cite{DimcaPapadimaSuciu} using the algebraic structure of the jump loci!

\subsection{Hodge theory}

Recall that Deligne \cite{DeligneII, DeligneIII} has constructed on the cohomology of algebraic varieties, or of quasi-Kähler varieties, a weight filtration $W_\bullet$ defined over $\QQ$ such that the weight-graded parts decompose like cohomology groups of compact Kähler manifolds; this abstract structure on cohomology is called \emph{mixed Hodge structure} (MHS). Put in family, this is the notion of \emph{variation of mixed Hodge structure}.

Then MHS have been constructed on many other topological invariants of these varieties, in particular on the rational homotopy groups by Morgan \cite{Morgan} and \cite{Hain}. A MHS on $\Ohat_\rho$ by Eyssidieux-Simpson \cite{EyssidieuxSimpson} when $X$ is \emph{compact}. We generalized this in \cite{Lefevre2} (representations with finite image) and then in \cite{Lefevre3} for all smooth varieties and representations that are the monodromy of a variation of mixed Hodge structure. In this case, the cohomology of $X$ with local coefficients also carries MHS: this due to Deligne, Zucker, Steenbrink\dots and ultimately Saito with his theory of mixed Hodge modules.

This, with the motivicity result of Esnault-Kerz, is a strong indication of the following result that we are now able to prove.

\begin{theorem}[{\S~\ref{section:geometry}}]
Let $W$ be am admissible variation of mixed Hodge structure over $X$. At a representation $\rho$ which is also the monodromy of variation of mixed Hodge structure, the ideals $J_k^i(W) \subset \Ohat_\rho$ defining $\Sigma_k^i(W) \subset \Hom(\pi_1(X,x), G)$ are sub-mixed Hodge structures.
\end{theorem}

As a corollary, we recover in rank one from this local structure theorem the sub-torus property. Remarkably, our proof uses only transcendental methods, for example we don't use the comparison with the moduli space of connections nor a Galois action of a field of definition of $X$ as is done in many of the previously cited woks. Hence it works for all smooth varieties, algebraic of quasi-Kähler, and without assumption that the local systems comes from geometry. Thus this also constitutes some new cases of the sub-torus property.

\begin{theorem}[{\S~\ref{section:application}}]
For $G=\CC^*$, the irreducible components of the relative loci $\Sigma_k^i(W)$ are translated sub-tori.
\end{theorem}

\subsection{Deformation functors}
The starting point of the methods that we use is the theory of Goldman-Millson \cite{GoldmanMillson}. To a representation $\rho$ is associated a local system $V_\rho$ and an adjoint local system $\ad_\rho\subset\End(V_\rho)$; there is also an augmentation $\epsilon$ evaluating sections of $\ad_\rho$ to the fiber at $x\in X$. The theory states that one can compute $\Ohat_\rho$ from the data of the differential graded (DG) Lie algebra $L$ of differential forms over $X$ with values in $\ad_\rho$, associating to it a deformation functor $\Def(L,\epsilon)$ over local Artin algebras and showing that this is isomorphic to the deformation functor associated to $\Ohat_\rho$ which is $\Hom(\Ohat_\rho, -)$. Without taking into account $\epsilon$, this corresponds to considering conjugacy classes of representations and $\Def(L,\epsilon)$ coincides with the classical deformation functor $\Def(\Ker(\epsilon))$ associated to the DG Lie algebra $\Ker(\epsilon)$, but this is not the right point of view.

This was extended by Budur-Wang \cite{BudurWang} to deal with the jump ideals: letting $M$ be the complex of differential forms over $X$ with values in $V_\rho$, then $M$ is a module (in the differential graded sense) over $L$ and to this is associated for every $i,k\subset\ZZ$ a sub-deformation functor
\begin{equation}
\Def_k^i(L,M,\epsilon)\subset\Def(L,M,\epsilon)
\end{equation}
which determines $J_k^i\subset\Ohat_\rho$. For the relative jump ideals $J_k^i(W)\subset\Ohat_\rho$ defining locally the relative jump loci $\Sigma_k^i(W)$, the module $M$ is the differential forms with values in $V_\rho \otimes W$.

In the cases we want to deal with, then $H(L)$ and $H(M)$ carry mixed Hodge structures (cohomology of $X$ with coefficients in an admissible variation of mixed Hodge structure). This already makes plausible an interaction between this MHS and the deformation functors. However, this MHS only exists \emph{on cohomology} whereas the deformation functor $\Def(L,\epsilon)$ (resp.\ $\Def_k^i(L,M,\epsilon)$) is determined by the whole of $L$ \emph{up to quasi-isomorphism} (resp.\ $(L,M)$ up to quasi-isomorphism of pairs).

Trying to understand this problem has lead to the notion of $L_\infty$ algebras, or \emph{Lie algebras up to homotopy}. These generalize DG Lie algebras but are equipped with \emph{higher operations} such that the cohomology of any DG Lie algebra $L$ becomes a $L_\infty$ algebra and $L$ becomes quasi-isomorphic to $H(L)$ \emph{as $L_\infty$ algebras}. Such theorems are called \emph{homotopy transfer of structures}. Furthermore, the deformation functors can also be written for $L_\infty$ algebras, and thus computed with $H(L)$ instead of $L$. We motivated this intensively and used this in our two previous articles \cite{Lefevre2, Lefevre3}.

At the same time and independently appeared the article of Budur-Rubi\'o \cite{BudurRubio} developing the theory of $L_\infty$ modules over $L_\infty$ algebras, showing a homotopy transfer theorem for the pair $(L,M)$ and defining deformation functors in this situation. There, the motivation besides Hodge theory is to compute these functors in \emph{finite-dimensional} vector spaces: typically $L,M$ are algebras of differential forms, so are infinite-dimensional, but their cohomology is finite-dimensional.

\subsection{Combination} Thus, the present works starts as a natural desire to combine \cite{BudurRubio} ($L_\infty$ algebras and modules, no MHS, rank $1$) with \cite{Lefevre2} ($L_\infty$ algebras and MHS, any rank, but no module). It seems indeed that all the ingredients are present: we have a pair $(L,M)$ carrying MHS on cohomology, determining sub-functors $\Def_k^i(L,M,\epsilon) \subset \Def(L,\epsilon)$ and the jump ideals $J_k^i \subset \Ohat_\rho$, we know how to put a MHS on $\Ohat_\rho$ and we want to show that the $J_k^i$ are sub-MHS.

This is indeed what we do through \S~\ref{section:proof-main}, where we do not need any more to refer to geometry but only to the abstract algebraic data that we get:

\begin{theorem}[{\S~\ref{section:proof-main}}]
Let $(L,M)$ be a pair which is at the same time
\begin{itemize}
\item A pair where $L$ is a differential graded Lie algebra and $M$ is a module over $M$,
\item A pair of mixed Hodge complexes, carrying mixed Hodge structures on the cohomology groups.
\end{itemize}
Let furthermore $\epsilon$ be an augmentation of $L$, injective on $H^0(L)$.

Then the complete local algebra $R$ that pro-represents $\Def(L, \epsilon)$ has a mixed Hodge structure, and the ideals $J_k^i \subset R$ that correspond to the sub-functors $\Def_k^i(L,M,\epsilon)$ are sub-mixed Hodge structures.
\end{theorem}

However there are several strong difficulties that appear \emph{already in the compact case} where we can since the beginning replace the pair $(L,M)$ up to quasi-isomorphism by its cohomology, and that we could not overpass until recently:
\begin{enumerate}
\item We want that $L$ be at the same time a mixed Hodge complex, carrying the filtrations inducing a MHS on its cohomology, and a DG Lie algebra at the level of cochain complex. This is not so trivial, the ideas come from rational homotopy theory starting with the work of Morgan \cite{Morgan} for the analogous problem with commutative DG algebras. We largely discussed this in \cite{Lefevre3} and can now use it directly in~\S~\ref{section:geometry} to deal with the additional data of a module over~$L$.

\item We have to deal with the functor $\Def(L,\epsilon)$, which is not exactly the same as the classical and much more studied $\Def(L)$. In \cite{Lefevre2} we claimed that the right point of view is to use a structure of $L_\infty$ algebra on the mapping cone of $\epsilon$; this was indeed our initial motivation for using $L_\infty$ algebras. In \S~\ref{section:cone-module} we now show a similar result that allows us to deal with the functors $\Def_k^i(L,M,\epsilon)$.

\item Many classical constructions of MHS go through constructing an appropriate mixed Hodge complex and showing that the object we want to put a MHS on comes from the cohomology of it. This is what we did in \cite{Lefevre2} for the above $R$, inspired notably by the theory of Hain \cite{Hain}. We could not follow this approach for the ideals $J_k^i$ and take here a completely different one. Instead, we show that we can can do a homotopy transfer of structure to replace $(L,M)$ by its cohomology such that it becomes quasi-isomorphic to it taking into account the MHS. This is possible thanks to a recent article of Cirici-Sopena \cite{CiriciSopena}.
\end{enumerate}

The last step requires again some care to deal with the augmentation. However, in the end what we get looks like a lot of  combinatorial and algebraic structures, carrying filtrations that define a MHS, but still is linear algebra in finite-dimensional vector spaces. From there, putting the MHS on the above $R$ and $J_k^i\subset R$ is not particularly difficult. This is the reason why we try to keep this article short: we recall the algebraic formulas that we need, but do not digress too much on the geometry since don't have much to add there, and then we come to the main proof.

\subsection{Acknowledgments}
I thank particularly Joana Cirici for discussions about her homotopy transfer theorem and Leonardo A.\ Lerer for discussions about the jump loci; both were essential to complete this work. I also thank Nero Budur and Marcel Rubi\'o for earlier discussions introducing me to this question.

\section{Formulas}
Here we recall what we need about deformation functors, DG Lie algebras, and $L_\infty$ algebras. We work over a fixed base field $\kk$ of characteristic zero and our deformations functors are always functors form the category $\Art$ of local Artin $\kk$-algebras (objects are denoted by $(A,\mathfrak{m}_A)$ where $\mathfrak{m}_A$ is the maximal ideal, and we assume that the residue field is $\kk$) to the category of sets.

\subsection{Cohomology jump ideals} See~\cite[\S~2]{BudurWang}.
Let $R$ be a ring and let $M$ be a complex of $R$-modules. We assume that $R$ is bounded-above and that its cohomology in each degree is finitely generated. Then there exists a bounded-above complex $F$ of finitely generated free $R$-modules with a quasi-isomorphism $F\rightarrow M$. The \emph{cohomology jump ideals} are ideals $J^i_k(M)\subset R$ depending on the two integers $i,k$ constructed as follows: choose bases for $F^i$ in each degree $i$ and write the matrix for the differential $d^i$ with coefficients in $R$, then $J^i_k(M)$ is the ideal generated by minors of size $\rank(F^i)-k+i$ of the matrix of
\begin{equation}
d^{i-1}\oplus d^i : F^{i-1} \oplus F^i \longrightarrow F^i \oplus F^{i+1} .
\end{equation}
These do not depend on the bases nor on the choice of $F$.

When $R$ is a field, then $J^i_k(M)=0$ if and only if $\dim(H^i(M))\geq k$.

\subsection{For DG Lie algebras} See~\cite[\S~3]{BudurWang}.
Let $L$ be a DG Lie algebra. It is equipped with a grading $L=\bigoplus_{i\in\ZZ} L^i$, a differential $d:L^i\rightarrow L^{i+1}$, and a bracket $[-,-]: L^i \otimes L^j \rightarrow L^{i+j}$ which is anti-symmetric, satisfies the graded Jacobi identity, and for which the differential is a derivation. To $L$ is associated a deformation functor $\Def(L)$ as follows. First there is the \emph{Maurer-Cartan functor} with for $(A,\mathfrak{m}_A)\in \Art$
\begin{equation}
\MC(L)(A) := \left\{ \omega\in L^1\otimes\mathfrak{m}_A\ \left|\
    d(\omega)+\frac{1}{2}[\omega,\omega] = 0 \right.\right\} .
\end{equation}
The nilpotent Lie algebra $L^0 \otimes \mathfrak{m}_A$ has a group structure denoted by
\begin{equation}
(\exp(L^0\otimes\mathfrak{m}_A), *)
\end{equation}
which acts on $\MC(L)(A)$ (the \emph{gauge action}). The deformation functor is the corresponding quotient:
\begin{equation}
\Def(L)(A) := \left\{ \omega\in L^1\otimes\mathfrak{m}_A\ \left|\
d(\omega)+\frac{1}{2}[\omega,\omega] = 0 \right.\right\} / \exp(L^0\otimes \mathfrak{m}_A) .
\end{equation}
The fundamental theorem is that $\Def(L)$ is invariant under a quasi-isomorphism of~$L$.    

Now let $M$ be a module over $L$, so we call $(L,M)$ a DG Lie pair. It has a grading $M=\bigoplus_{i\in\ZZ} M^i$, a differential $d:M^i \rightarrow M^{i+1}$ and an action map $L^i \otimes M^j \rightarrow M^{i+j}$. We need to assume that $M$ is bounded above and has finite-dimensional cohomology in each degree. Then to $(L,M)$ and to any pair of integers $i,k$ is associated a \emph{deformation functor with constraints} $\Def_k^i(L,M)$ which is a sub-functor of $\Def(L)$. It is constructed as follows: for $(A,\mathfrak{m}_A)\in \Art$, an element $\omega\in \MC(L)(A)$ defines a differential $d_\omega$ on $M\otimes A$ by
\begin{equation}
d_\omega(-) := {d(-)}\otimes{\id_A} + \omega\otimes -
\end{equation}
using the action of $\omega\in L^1\otimes \mathfrak{m}_A$ on $M\otimes A$.
The \emph{cohomology jump ideals} of $M\otimes A$ with respect to this $d_\omega$
\begin{equation}
J_k^i(M\otimes A, d_\omega) \subset A
\end{equation}
are then canonically defined ideals of $A$ and
\begin{equation}
\Def_k^i(L,M)(A) :=  \left\{ \omega\in L^1\otimes\mathfrak{m}_A\ \left|\
    d(\omega)+\frac{1}{2}[\omega,\omega] = 0,\ J_k^i(M\otimes A, d_\omega)=0 \right.\right\} / \exp(L^0\otimes \mathfrak{m}_A) .
\end{equation}
These sub-functors are invariant under a quasi-isomorphism of the pair $(L,M)$.

\subsection{With augmentations} See~\cite[\S~2]{Lefevre2} and~\cite[\S~5]{BudurWang}.
In practice, our main interest is not directly the functor $\Def(L)$ but a small modification of it taking into account an augmentation $\epsilon:L\rightarrow \mathfrak{g}$ to a Lie algebra (finite-dimensional, seen as a DG Lie algebra concentrated in degree zero). This is this functor that, in our problem, controls the deformation theory of representations of the fundamental group, seen also as the deformation theory of local systems with a fixed framing at the base point.

So let $\epsilon:L\rightarrow \mathfrak{g}$ be an augmentation of $L$. The corresponding \emph{augmented deformation functor} is defined as: for $(A,\mathfrak{m}_A)\in\Art$
\begin{equation}
\Def(L,\epsilon)(A) :=  \left\{ (\omega,e^\alpha) \in (L^1\otimes\mathfrak{m}_A)
  \times \exp(\mathfrak{g}\otimes\mathfrak{m}_A)\ \left|\
    d(\omega)+\frac{1}{2}[\omega,\omega]=0 \right.\right\}/\exp(L^0 \otimes\mathfrak{m}_A)
\end{equation}
where $e^\lambda \in \exp(L^0\otimes \mathfrak{m}_A)$ acts via
\begin{equation}
e^\lambda \cdot (\omega, e^\alpha) := (e^\lambda \cdot x, e^\alpha*e^{-\epsilon(\lambda)}) .
\end{equation}
Similarly one defines for an \emph{augmented DG Lie pair} $(L,M,\epsilon)$ (a pair $(L,M)$ and an augmentation of $L$)
\begin{multline}
\Def^i_k(L,M,\epsilon)(A) := \bigg\{ (\omega,e^\alpha) \in (L^1\otimes\mathfrak{m}_A)
\times \exp(\mathfrak{g}\otimes\mathfrak{m}_A)\ \bigg| \\ 
d(\omega)+\frac{1}{2}[\omega,\omega]=0,
\ J_k^i(M\otimes A, d_\omega)=0\bigg\}/\exp(L^0 \otimes\mathfrak{m}_A).
\end{multline}
The functor $\Def(L,\epsilon)$ (resp.\ $\Def^i_k(L,M,\epsilon)$) is invariant under a quasi-isomorphism of $L$ (resp.\ of the pair $(L,M)$) that commutes with the augmentation to $\mathfrak{g}$.

\subsection{For \texorpdfstring{$L_\infty$}{L-infinity} algebras}
See~\cite{BudurRubio}.
An $L_\infty$ algebra is given by a graded vector space $C=\bigoplus_{i\in\ZZ} C^i$ and sequence of anti-symmetric linear maps $\ell_n:C^{\otimes n} \rightarrow C$ for each $n\geq 1$, where $\ell_n$ is of degree $2-n$, satisfying an infinite list of axioms. The first of these states that $\ell_1\circ\ell_1=0$ (so $\ell_1$ can be identified with a differential $d$), the second one that $\ell_1$ is a derivation for $\ell_2$.

It is more practical to encode these axioms as a codifferential on a coalgebra. Let $\mathscr{C}(C)$ by $\Sym(C[1])$ with its graded cofree cocommutative coalgebra structure; the operations $\ell_n$ correspond to components
\begin{equation}
q_n:\Sym^n(C[1])\rightarrow C[1]
\end{equation}
of a coderivation $Q$ (of degree $1$) on the coalgebra $\mathscr{C}(C)$, and the axioms are all encoded in the relation $Q\circ Q=0$. We then think of $\mathscr{C}$ as a functor from $L_\infty$ algebras to the category of DG coalgebras (not \emph{all} such coalgebras: the ones that are dual to complete augmented algebras).

Then the Maurer-Cartan functor of $C$ is given by: for $(A,\mathfrak{m}_A)\in \Art$
\begin{equation}
\MC(C)(A) := \Hom_{{\mathbf{CoAlg}}}(A^*, \mathscr{C}(C)) 
= \left\{ \omega \in C^1\otimes\mathfrak{m}_A\ \left|\  \sum_{n\geq 1} \frac{\ell_n(\omega,\dots,\omega)}{n!}=0 \right.\right\}
\end{equation}
where the $\Hom$ is taken in the appropriate category of DG coalgebras.
The deformation functor of $L$ is simply the quotient by homotopy:
\begin{equation}
\Def(C)(A) :=  \MC(C)(A)/\sim
\end{equation}
with $\omega_0\sim\omega_1$ if and only if there exists an element
\begin{equation}
\omega\in \MC(C)(A\otimes \kk[t,dt])
\end{equation}
(where $\kk[t,dt]$ is the algebra of polynomial differential forms on $\kk$) with $\omega(t=0)=\omega_0$ and $\omega(t=1)=\omega_1$. This is known to coincide with the previously defined $\Def(C)$ if $C$ is a DG Lie algebras (that is, $\ell_n=0$ for $n>3$): we see that the Maurer-Cartan elements are the same, and two Maurer-Cartan elements are homotopically equivalent if and only if they are gauge equivalent.

There is also a notion of $L_\infty$ module over $C$. Such a module $M$ is by definition equipped with a grading $M=\bigoplus_{i\in\ZZ} M^i$ and \emph{module action maps}
\begin{equation}
m_n:C^{\otimes (n-1)} \otimes M \longrightarrow M, \quad n\geq 1
\end{equation}
where $m_n$ has degree $2-n$, satisfying an infinite list of axioms; the first one states that $m_1 \circ m_1=0$ (so $m_1$ can be identified with a differential on $M$). The pair $(C,M)$ is called a \emph{$L_\infty$ pair}.

For $A\in \Art$, a Maurer-Cartan element $\omega\in \MC(C)(A)$ defines a differential $d_\omega$ on $M\otimes A$ by
\begin{equation}
\label{equation:derivation-M}
d_\omega(-):=\sum_{n\geq 0}\frac{1}{n!}(m_{n+1}\otimes\id_A)(\omega^{\otimes n} \otimes -)
\end{equation}
and defines, as always under the hypothesis that $M$ is bounded-above and its cohomology is finite-dimensional in each degree, jump ideals $J_k^i(M\otimes A, d_\omega)$ with respect to $d_\omega$. Then there are sub-functors of $\Def(C)$
\begin{equation}
\Def_k^i(C,M)(A) := \left\{ \omega \in \MC(C)(A) \left|\
J_k^i(M\otimes A, d_\omega)=0 \right.\right\}/\sim 
\end{equation}
where $\sim$ is again homotopy equivalence.

For $L_\infty$ algebras there is a notion of $L_\infty$ (or \emph{weak}) morphism; it can be defined as a morphism of DG coalgebras $\mathscr{C}(C)\rightarrow \mathscr{C}(C')$. It is called \emph{weak equivalence} if its linear component, the morphism of complexes $C\rightarrow C'$, is a quasi-isomorphism. Similarly there is a notion of weak morphisms and weak equivalences of $L_\infty$ pairs. The \emph{homotopy transfer theorem} states that any $L_\infty$ algebra is weakly equivalent to its cohomology which receives a \emph{transferred} $L_\infty$ structure, and similarly a $L_\infty$ pair becomes weakly equivalent to its cohomology pair with a transferred $L_\infty$ structure. The fundamental theorem is that the functors $\Def(C)$, $\Def_k^i(C,M)$ are invariant under weak equivalences.

\subsection{From \texorpdfstring{$L_\infty$}{L-infinity} pairs to \texorpdfstring{$L_\infty$}{L-infinity} algebras}
\label{section:L-infinity-pair-sum}
Finally we will need one practical construction to deduce theorems for $L_\infty$ pairs from known theorems for $L_\infty$ algebras, described in \cite[\S~3.3--3.4]{BudurRubio}). From $(C,M)$ one can put a $L_\infty$ structure $(j_n)$ on the graded vector space $C\oplus M$, with for $(a_1,\xi_1),\dots,(a_n,\xi_n)\in C\oplus M$
\begin{equation}
\label{equation:L-infinity-pair-sum}
j_n((a_1,\xi_1),\dots,(a_n,\xi_n)) :=  \left(
\ell_n(a_1,\dots,a_n), \sum_{i=1}^n (-1)^{n-i+|\xi_i|\sum_{k=i+1}^n |a_k|} m_n(a_1, \dots, \hat{a}_i, \dots,a_n,\xi)
\right).
\end{equation}
The interest of this construction is that given the complexes $C,M$ and this $L_\infty$ structure on $C\oplus M$ one can recover the $L_\infty$ structure of $C$ and the module action maps; furthermore a morphism of $L_\infty$ pairs $(C,M)\rightarrow (C',M')$ induces a morphism of $L_\infty$ algebras $C\oplus M\rightarrow C'\oplus M'$ from which one can recover the morphism of $L_\infty$ pairs, and it is a weak equivalence of $L_\infty$ algebras if and only if the morphism of pairs is a weak equivalence.

This is actually how Budur-Rubi\'o deduce their theorem of transfer of structure for $L_\infty$ pairs from the one for $L_\infty$ algebras.

\subsection{The cone}
\label{section:the-cone}
See~\cite[\S~5]{Lefevre2}. Given an augmented DG Lie algebra $\epsilon:L\rightarrow \mathfrak{g}$, Fiorenza and Manetti defined in~\cite{FiorenzaManetti} a canonical structure of $L_\infty$ algebra on the shifted mapping cone $C$ of $\epsilon$, which has $C^i=L^i$ for $i\neq 1$ and $C^1=L^1\oplus \mathfrak{g}$, with differential induced by the one of $L$ and
\begin{equation}
d^0_C(a):=(d_L^0(a), \epsilon(a)).
\end{equation}
We will call this simply the \emph{$L_\infty$ cone}. The crucial remark is that for this structure there is an equality between the deformation functors
\begin{equation}
\Def(L, \epsilon) = \Def(C)
\end{equation}
where the left one is defined via gauge transformations and the right one via a $L_\infty$ structure.

Our goal in \S~\ref{section:cone-module} will be to generalize this to the deformation functors with constraints, so we need the formulas.

The operation $\ell_2$ is induced by the Lie bracket and we need only to describe its component $C^0\otimes C^1 \rightarrow C^1$: it is
\begin{equation}
\ell_2(a,(b,v)) := \left([a,b],\ \frac{1}{2}\,[\epsilon(a),v]\right), \quad a\in L^0,b\in L^1, v\in\mathfrak{g} .
\end{equation}
Finally there is a non-trivial operation $\ell_n$ for $n\geq 3$. Applied to $n$ elements of $\mathfrak{g}$ and $k>1$ elements of $L$, $\ell_{n+k}$ is zero. It is non-trivial only when applied to $n$ elements $u_1,\dots,u_r$ of $\mathfrak{g}$ and one element $a$ of $L^0$, with values in $\mathfrak{g}\subset C^1$, given by a formula
\begin{equation}
\ell_{n+1}(u_1,\dots,u_n,a)=B_n\sum_\tau \pm [u_{\tau(n)}, [u_{\tau(2)},\dots,
[u_{\tau(n)}, \epsilon(a)] \dots ]]
\end{equation}
where the sum is over the symmetric group and with a constant $B_n$ and unimportant (for us) signs in the sum.

\section{Pro-representability of deformation functors}
\label{section:pro-representability}

We want to recall precisely how and when the deformation functors are pro-representable. Recall that the functor $F$ on $\Art$ is said to be pro-representable if there exists an element $R\in \widehat{\Art}$ such that $F$ is isomorphic to
\begin{equation}
A\in\Art \longmapsto \Hom_{\widehat{\Art}}(R,A).
\end{equation}
This $R$ is then unique up to isomorphism.

In this section, we work with a $L_\infty$ algebra $C$, again over any field $\kk$ of characteristic zero.

\begin{lemma}
\label{lemma:homotopy-trivial}
Let $C$ be a $L_\infty$ algebra. Assume that $C^{< 0}=0$ and $C^0=0$. Then the homotopy relation among Maurer-Cartan elements of $C$ is trivial.
\end{lemma}

\begin{proof}
Indeed, for $(A,\mathfrak{m}_A)\in \Art$ take a Maurer-Cartan element
\begin{equation}
\omega\in\MC(C)(A \otimes \kk[t,dt])
\end{equation}
and write it as $\omega=\gamma(t)+a(t)dt$ where $a(t)$ is a polynomial with values in $C^0\otimes\mathfrak{m}_A$ and $\gamma(t)$ is a polynomial with values in $C^1\otimes\mathfrak{m}_A$. Here $a=0$, and then the Maurer-Cartan equation translates into
\begin{equation}
\left( d(\gamma(t)) - \frac{\partial\gamma}{\partial t}dt \right) + \sum_{n\geq 2}\frac{\ell_n(\gamma(t),\dots,\gamma(t))}{n!}
\end{equation}
which tells us that $\gamma(t)$ is constant and is a Maurer-Cartan element of $C\otimes \mathfrak{m}_A$.
\end{proof}

Under the previous hypothesis, it is easy to pro-represent the functor $\Def(C)$ and functorially in $C$.

\begin{corollary}
\label{corallary:pro-representability}
If furthermore $C^1$ is finite-dimensional, then $\Def(C)$ is pro-representable. If $M$ is an $L_\infty$ module over $C$, bounded above and finite-dimensional in each degree, then $\Def_k^i(C,M)$ is pro-representable.
\end{corollary}

Let us explain how this works. For $A\in \Art$, because of the previous lemma,
\begin{equation}
\Def(C)(A)=\MC(C)(A)=\Hom_{\mathbf{CoAlg}}(A^*, \mathscr{C}(C))
\end{equation}
but since $C$ is graded only in degree $\geq 1$ then $\mathscr{C}(C)$ is non-negatively graded and this is also
\begin{equation}
\Def(C)(A)=\Hom_{\mathbf{CoAlg}}(A^*, H^0(\mathscr{C}(C))) .
\end{equation}
Here $H^0(\mathscr{C}(C))$ has a canonical increasing filtration by finite-dimensional sub-coalgebras, coming from the canonical filtration of $\mathscr{C}(C)$ by the number of tensor powers, which dualizes to a tower of finite-dimensional quotients of a pro-Artin algebra. Hence we can dualize and write
\begin{equation}
\Def(C)(A)=\Hom_{\widehat{\Art}}(R, A), \quad R=H^0(\mathscr{C}(C))^* .
\end{equation}

In a more concrete way, the Maurer-Cartan equation gives formal power series equations in the finite-dimensional vector space $C^1$ with values in $C^2$. This defines the complete local algebra $R$ with finite-dimensional tangent space (i.e.\ $R$ is pro-Artin) as a quotient of the completed symmetric algebra on $(C^1)^*$.

Once constructed $R$, there is a \emph{universal Maurer-Cartan element} $\omega_u \in \MC(C)(R)$. It has the universal property that for any $A\in\Art$, then any $\omega\in \Def(C)(A)$ is induced by $\omega_u$ via a unique map $R\rightarrow A$ in $\widehat{\Art}$. With this one can form the \emph{universal derivation} $d_u:=d_{\omega_u}$ on the complex of free finitely generated $R$-modules $M\otimes R$ and then the jump ideals $J_k^i(M\otimes R, d_{u}) \subset R$. To do this properly one actually has to truncate everything to order $n$ and work with the ring $R_n:=R/\mathfrak{m}_R^n$, which represents the restriction of $\Def(C)$ to those Artin algebras $A$ with $\mathfrak{m}_A^n=0$, and the reduction modulo ${\mathfrak{m}_R^n}$ of $\omega_u$ is the universal Maurer-Cartan element for this. The ideal $J_k^i(M\otimes R, d_{u})\subset R$ is then the projective limit of the $J_k^i(M\otimes R_n, d_{u})\subset R_n$.

Then $\Def_k^i(C,M)$ is pro-represented by the algebra
\begin{equation}
R / J_k^i(M\otimes R, d_{u}) .
\end{equation}

The functoriality in $C$ of this construction appears clearly, because a morphism (even a weak morphism) of $L_\infty$ algebras $C\rightarrow C'$ induces a morphism of coalgebras $\mathscr{C}(C)\rightarrow \mathscr{C}(C')$ and so a morphism at the level of~$R$. The invariance under weak equivalence also follows because weak equivalences induce a quasi-isomorphism at the level of the coalgebra $\mathscr{C}(C)$. Once fixed $R$, the jump ideals are canonically defined as ideals of $R$.

\begin{corollary}
The above corollary holds if the hypothesis are satisfied only at the level of cohomology: $C$ has $H^{\leq 0}(C)=0$ and $H^1(C)$ finite-dimensional, $M$ is bounded above and its cohomology is finite-dimensional in each degree.
\end{corollary}

This follows from the homotopy transfer of structure and the invariance of the deformation functor under quasi-isomorphisms, as in \cite{BudurRubio} for $L_\infty$ pairs.

\section{Augmented deformation functors with constraints}
\label{section:cone-module}
In our previous work \cite{Lefevre2}, one important starting point (easy remark, but conceptually this is one motivation for going to $L_\infty$ algebras) is that for the $L_\infty$ cone $C$ over an augmented DG Lie algebra $\epsilon: L\rightarrow \mathfrak{g}$ then we recover the deformation functor $\Def(C)=\Def(L,\epsilon)$. We need a version of this for augmented $L_\infty$ pairs.

\begin{theorem}
\label{theorem:module-over-cone}
Let $\epsilon:L\rightarrow\mathfrak{g}$ be an augmented DG Lie algebra and let $M$ be a module over $L$. Let $C$ be the $L_\infty$ cone over $\epsilon$. There exists a structure of $L_\infty$ module of $M$ over $C$ such that we recover the deformation functor
\begin{equation*}
\Def_k^i(C,M) = \Def_k^i(L,M,\epsilon).
\end{equation*}
This structure is functorial in the augmented pair $(L,M,\epsilon)$.
\end{theorem}

\begin{proof}
First, to show that such a structure exists, it is practical to form the DG Lie algebra  $L\oplus M$ as in \S~\ref{section:L-infinity-pair-sum}. The bracket is simply described as
\begin{equation}
[(\ell_1,\xi_1), (\ell_2,\xi_2)] := ([\ell_1,\ell_2], \ell_1.\xi_2 - \ell_2.\xi_1) .
\end{equation}
Then there is an augmentation $\epsilon':L\oplus M \rightarrow \mathfrak{g}$ induced by $\epsilon$ which is again a morphism of DG Lie algebras (it is in fact induced by a morphism of pairs $(L,M)\rightarrow (\mathfrak{g},0)$). Form its $L_\infty$ cone $C'$: as complex, $C'= L \oplus M \oplus \mathfrak{g}[-1] = C \oplus M$. In the formulas of Fiorenza-Manetti, the only non-trivial higher operations on $C'$ are made with iterated brackets in $\mathfrak{g}$ of $n$ elements of $\mathfrak{g}$ with $\epsilon'$ of one element of $C'$. But this $\epsilon'$ is zero on $M$. So with this $L_\infty$ structure on $C\oplus M$ we recover, first a $L_\infty$ structure on $C$ which is the same as the one of $L_\infty$ cone over $\epsilon$, and then a structure of $L_\infty$ module on $M$. And we see that the action of $C$ on $M$ is induced only by the action of $L$: $\mathfrak{g}$ does not act on $M$.

Then we identify the deformation functor. Let $(A,\mathfrak{m}_A)\in \Art$. An element of $\Def_k^i(C,M)(A)$ is given by an element $\omega' \in \MC(C)(A)$ such that $J_k^i(M\otimes A, d_{\omega'})=0$. We know then from the computations of Fiorenza-Manetti that $\omega'$ corresponds to a pair of $\omega\in \MC(L)(A)$ and $e^{\alpha}\in \exp(\mathfrak{g}\otimes \mathfrak{m}_A)$. And again since the action of $\mathfrak{g}$ on $M$ is trivial we see that that for the induced differential
\begin{equation}
d_{\omega'} = d_{\omega}
\end{equation}
so of course $J_k^i(M\otimes A, d_{\omega'}) = J_k^i(M\otimes A, d_{\omega})$.

Finally, we use also that the homotopy relation between two objects $(\omega_1, e^{\alpha_1})$ and $(\omega_2, e^{\alpha_2})$ of $\MC(C)(A)$ is already computed by Fiorenza-Manetti. It is given by the existence of some $e^\lambda \in \exp(L^0 \otimes \mathfrak{m}_A)$ such that $e^\lambda.\omega_1=\omega_2$ and $e^{\alpha_2}=e^{\alpha_1}*e^{-\epsilon(\lambda)}$. So we see that this gives the same equivalence relation as the one defining $\Def_k^i(L,M,\epsilon)$ by gauge action.
\end{proof}

The functoriality also gives a $L_\infty$ proof of the invariance of the augmented deformations functors: a quasi-isomorphism of pairs $\phi:(L_1,M_1)\rightarrow (L_2,M_2)$ commuting with the augmentations $\epsilon_1,\epsilon_2$ to the same $\mathfrak{g}$ induces a quasi-isomorphism of $L_\infty$ pairs $(C_1,M_1)\rightarrow (C_2,M_2)$ hence $\Def_k^i(L_1,M_1,\epsilon_1)=\Def_k^i(L_2,M_2,\epsilon_2)$.

Up to now, we have exposed everything we need to pro-represent the deformation functor $\Def_k^i(L,M,\epsilon)$ in our case of interest where $L$ is only non-negatively graded and the restriction of $\epsilon$ to $H^0(L)$ is injective: then we form the $L_\infty$ cone $C$ which has $H^0(C)=0$ and turn interest to the deformation functor $\Def_k^i(C,M)$, applying the methods of the previous section, and eventually using a homotopy transfer of structure we reduce to the $L_\infty$ pair $(H(C),H(M))$.

\section{Mixed Hodge structures}
In this section, $\kk$ is a sub-field of $\RR$. We could also work with $\kk=\CC$ by appropriately modifying the definitions.

\subsection{Mixed Hodge structures}
First recall what a mixed Hodge structure is.

\begin{definition}
A \emph{mixed Hodge structure} (MHS) is given by a finite-dimensional vector space $H$ over $\kk$, with an increasing filtration $W_\bullet$ of $H$ (\emph{weight} filtration) and a decreasing filtration $F^\bullet$ of $H\otimes\CC$ such that the weight-graded pieces decompose as
\begin{equation}
\Gr^W_k(H) \otimes \CC = W_k(H)/W_{k-1}(H)\otimes \CC = \bigoplus_{p+q=k} H^{p,q}
\end{equation}
where the complex conjugation exchanges $H^{p,q}$ and $H^{q,p}$ and for the induced filtration $F^\bullet$ on $\Gr^W_k(H)\otimes\CC$,
\begin{equation}
F^p=\bigoplus_{p'\geq p} H^{p',k-p'}.
\end{equation}
A MHS where $W_\bullet$ is concentrated in weight $k$ is called a \emph{pure Hodge structure} of weight $k$
\end{definition}

The main theorem of Deligne is that MHS exist on each cohomology group of complex algebraic varieties or quasi-Kähler manifolds; when $X$ is proper and smooth, or compact Kähler, then on $H^k(X,\QQ)$ the MHS is the classical pure Hodge structure of weight $k$.

MHS form a nice abelian category with a tensor product. Hence it is easy to speak of algebras, Artin rings, graded Lie algebras etc internally to the category of MHS. We will in particular denote by $\mathbf{Art+MHS}$ the category of Artin local algebras with a compatible MHS, and by $\mathbf{\widehat{Art}+MHS}$ the pro-Artin algebras $R$ all whose quotients $R/\mathfrak{m}_R^n$ are in $\mathbf{Art+MHS}$ in a compatible way (since MHS are by definition finite-dimensional, $R$ cannot carry a MHS in the usual sense, so it is really a projective limit of objects of $\mathbf{Art+MHS}$).

\subsection{Mixed Hodge complexes}
In practice, MHS are not easy to construct. The notion introduced by Deligne is that of a mixed Hodge complex, whose cohomology gives MHS in each degree.

\begin{definition}
\label{definition:MHC}
A \emph{mixed Hodge complex} (MHC) is given by two bounded-below complexes $K_\kk$ over $\kk$, carrying a filtration $W_\bullet$, and $K_\CC$ over $\CC$ carrying two filtrations $W_\bullet,F^\bullet$, with finite-dimensional cohomology in each degree, and a quasi-isomorphism $\phi$ between $K_\kk \otimes \CC$ and $K_\CC$, such that:
\begin{itemize}
\item $\phi$ is given by a zig-zag of $W$-filtered quasi-isomorphisms (i.e.\ inducing a quasi-isomorphism on the $W$-graded pieces),
\item On $\Gr^W_k(K_\CC)$ the induced differential is strictly compatible with $F$,
\item On $H^n(\Gr^W_k(K_k))\otimes \CC \simeq H^n(\Gr^W_k(K_\CC))$, the filtration $F$ induces a pure Hodge structure of weight $k+n$.
\end{itemize}
\end{definition}

As we said, the theorem of Deligne is that the $W$-spectral sequence of $K$ degenerates at $E_2$ and then each $H^n(K_\kk)$ is a MHS over $\kk$, for the shift $W[n]$ of the induced filtration $W$ such that $W[n]_k=W_{k-n}$ and the induced filtration $F$ over $\CC$.

Let us also mention that a \emph{quasi-isomorphism} of MHC $K\rightarrow L$ is given by $W$-filtered quasi-isomorphisms $K_\kk \rightarrow L_\kk$ and by a bifiltered quasi-isomorphism (inducing a quasi-isomorphism on the $\Gr^w \Gr_F$ pieces) $K_\CC \rightarrow L_\CC$.

Again MHC admit a tensor product, so can speak of algebras internally to the category of MHC. For example in a MHC with a structure of DG Lie algebras, both $K_\kk$ and $K_\CC$ are DG Lie algebras, the filtrations are multiplicative under the Lie bracket and the quasi-isomorphisms between them are quasi-isomorphisms of DG Lie algebras. We call these simply MHC + Lie algebras. Similarly there are pairs of MHC with a structure of DG Lie pair, called MHC + Lie pairs.

There is however a problem with cones and $L_\infty$ algebras. The axioms of MHC indeed imply that the shift $K[r]$ of the MHC $K$ (with $K[r]^n:=K^{n+r}$) is again a MHC if one puts on it the induced filtration $F$, with
\begin{equation}
F^p (K_\CC[r]^n) := F^p K^{n+r}_\CC,
\end{equation}
and the shifted weight filtration $W[r]$ so that
\begin{equation}
W_k (K[r]^n) := W_{k-r} K^{n+r}.
\end{equation}
Similarly, the shifted cone $C$ of a morphism $K\rightarrow L$ of MHC is again a MHC with
\begin{equation}
W_k C^n := W_k K^n \oplus W_{k+1} L^{n-1}.
\end{equation}
Thus one has to be very careful if one wants to consider operations of $L_\infty$ algebras as in \cite[\S~8]{Lefevre2}, that shift the degree, on MHC.

\subsection{Absolute Hodge complexes}
\label{section:AHC}
To remedy this, we will work with a category that differs slightly of MHC. This was introduced first by Beilinson (\cite{Beilinson}) in order to study the derived category of MHS, see also \cite{CiriciGuillenComplexes} for many more details.

\begin{definition}[{\cite[Def.~4.5]{CiriciGuillenComplexes}}]
An \emph{absolute Hodge complex} (AHC) is given, as for MHC, by filtered complexes $(K_\kk,W_\bullet)$ and $(K_\CC,W_\bullet,F^\bullet)$ a quasi-isomorphism $\phi$ between $K_\kk \otimes \CC$ and $K_\CC$, such that:
\begin{itemize}
\item $\phi$ is given by a zig-zag of $W$-filtered quasi-isomorphisms,
\item On $\Gr^W_k(K_\CC)$ the induced differential is strictly compatible with $F$,
\item On $H^n(\Gr^W_k(K_k))\otimes \CC \simeq H^n(\Gr^W_k(K_\CC))$, the filtration $F$ induces a pure Hodge structure of weight~$n$.
\end{itemize}
\end{definition}

With this definition one can show that each term $H^n(K)$ carry a MHS directly with the induced filtrations $F$ and $W$: they play a more symmetrical role and need not to be shifted.

The definition is also made up so that one gets an AHC from a MHC by applying the following functor $\Dec$ to the weight filtration.

\begin{definition}
For any filtered complex $(K^\bullet,W_\bullet)$ the filtration $\Dec W$ is defined by
\begin{equation}
(\Dec W)_k K^n := \Ker\big( W_{k-n} K^n \stackrel{d}{\longrightarrow} K^{n+1}/W_{k-n-1} \big) .
\end{equation}
\end{definition}

It is easy to show that on cohomology
\begin{equation}
\Gr^{\Dec W}_k H^n(K) = \Gr^W_{k-n} H^n(K)
\end{equation}
and that $\Dec$ shifts the associated $W$-spectral sequence.
By applying $\Dec$ to the filtration $W$ of the two components of a MHC one gets a functor:

\begin{lemma}[{\cite[\S~4.2]{CiriciGuillenComplexes}}]
The functor $\Dec$ sends MHC to AHC.
\end{lemma}

Furthermore the mapping cone $C$ of a morphism of AHC $K\rightarrow L$ is again a AHC with
\begin{equation}
F^p C^n_\CC := F^p K^n_\CC \oplus F^p L^{n-1}_\CC
\end{equation}
and
\begin{equation}
W_k C^n := W_k K^n \oplus W_k L^{n-1},
\end{equation}
and also the shift of an AHC is again an AHC in the most obvious way. The tensor product still makes sense as for MHC, and the functor $\Dec$ is \emph{lax symmetric monoidal} (see \S~\ref{section:with-MHS}): for two MHC $K,L$ there is a natural map of AHC
\begin{equation}
\label{equation:dec-symmetric-monoidal}
(\Dec K) \otimes (\Dec L) \longrightarrow \Dec (K \otimes L)
\end{equation}
and this implies that the $\Dec$ of a MHC + Lie pair is an AHC + Lie pair.

\subsection{Tannakian point of view}
\label{section:tannakian}
Let us recall briefly that the category of MHS is a \emph{Tannakian category} whose fiber functor is given by the underlying $\kk$-vector space. In particular it is equivalent to the category of representations of a certain algebraic group $\mathscr{G}$. In this setting, a sub-MHS of $H$ is a sub-representation of $\mathscr{G}$, that is, a sub-space of $H$ stable by the $\mathscr{G}$-action. Morphisms of MHS are $\mathcal{G}$-equivariant maps. An algebra $(A,\mu)$ with a compatible MHS means that the product map is $\mathcal{G}$-equivariant, that is, for $g\in\mathcal{G}$ and $a,b\in A$ one has 
\begin{equation}
g\cdot\mu(a,b)=\mu(g\cdot a, g\cdot b) ;
\end{equation}
similarly if $M$ is an $A$-module and a MHS then the action map $m:A\otimes M\rightarrow M$ respects the MHS iff it satisfies that for $g\in\mathcal{G}$, $a\in A$, $\xi\in M$ then
\begin{equation}
g\cdot m(a,\xi)=m(g\cdot a, g\cdot \xi).
\end{equation}
We also see that if the element $a\in A$ is fixed by the $\mathscr{G}$-action (that is, is defined over $\kk$ and is in $W_0 \cap F^0$) then the multiplication by $a$ in $M$ is equivariant, that is, a morphism of MHS.

\section{Geometry}
\label{section:geometry}
Here we mostly recall the framework of \cite{Lefevre2}, \cite{Lefevre3} incorporating the jump ideals.

\subsection{Representations}
Let $X$ be any smooth differentiable manifold having the homotopy type of a finite CW complex. Let $G$ be a linear algebraic group over $\CC$ with a fixed embedding into some $GL_N(\CC)$. We want to study representations of the fundamental group $\pi_1(X,x)$ (at a fixed base point $x$) into $G$. The set
\begin{equation}
\Hom_{\mathbf{Group}}(\pi_1(X,x), G(\CC))
\end{equation}
has naturally a structure of affine scheme of finite type over $\CC$ because $G$ is affine and $\pi_1(X,x)$ is finitely generated; it is called the \emph{representation variety} and is denoted by $\Hom(\pi_1(X,x), G)$. Inside of it are defined \emph{closed} subschemes corresponding to the representations $\rho$ where the dimension of the cohomology of the corresponding local system $V_\rho$ ($\rho$ acts on some $\CC^N$ via the fixed linear embedding of $G)$ jumps: for two integers $i,k\geq 0$
\begin{equation}
\Sigma_k^i := \left\{ \rho:\pi_1(X,x)\rightarrow G(\CC)\ \big|\ \dim(H^i(X, V_\rho)) \geq k \right\}
\subset \Hom(\pi_1(X,x), G).
\end{equation}
These are the \emph{cohomology jump loci}. More generally, we can fix a local system $W$ over $X$ and consider the \emph{relative} cohomology jump loci
\begin{equation}
\Sigma_k^i(W) := \left\{ \rho:\pi_1(X,x)\rightarrow G(\CC)\ \big|\ \dim(H^i(X, V_\rho \otimes W)) \geq k \right\}.
\end{equation}
For $W$ the trivial local system $\underline{\CC}_X$ this recovers the absolute loci.

\begin{definition}
Let $\rho:\pi_1(X,x)\rightarrow G(\CC)$ be a given representation, seen as a point of the scheme $\Hom(\pi_1(X,x), G)$. We denote by $\Ohat_\rho$ the completion of the local ring of $\Hom(\pi_1(X,x), G)$ at $\rho$ and by $J_k^i(W) \subset \Ohat_\rho$ the ideal defining $\Sigma_k^i(W)$ at $\rho$. This is the \emph{jump ideal} of $\rho$ relatively to $W$. The associated \emph{deformation functors} $\Art \rightarrow \Set$ are given by
\begin{equation}
\Def(\rho):A \longmapsto \Hom(\Ohat_\rho, A)
\end{equation}
and
\begin{equation}
\Def_k^i(\rho,W): A \longmapsto \Hom(\Ohat_\rho/J_k^i(W), A) .
\end{equation}
\end{definition}

Indeed 
\begin{equation}
\Def(\rho)(A) = \left\{ \tilde\rho:\pi_1(X,x) \rightarrow G(A)\ \big|\ \tilde\rho=\rho \ \text{in}\ G(\CC)\right\}
\end{equation}
is the functor of formal deformations of $\rho$ and $\Def_k^i(\rho,W)$ is its functor of deformations preserving the cohomology constraints.

Finally let us mention that everything of this can be done over a sub-field $\kk\subset \CC$: if~$G$ is defined over~$\kk$ and $\rho$ has values in $G(\kk)$, then the schemes $\Hom(\pi_1(X,x),G)$, $\Sigma_k^i$ are defined over~$\kk$ and the rings $\Ohat_\rho$, $\Ohat_\rho/J_k^i$ are also defined over~$\kk$; if $W$ is a local system over $\kk$ then the $\Sigma_k^i(W)$ are also defined over $\kk$. The deformation functors are then defined over local Artin algebras over~$\kk$.

\subsection{Goldman-Millson theory}
Now we fix the representation $\rho$ and the local system $W$. To this is associated, as we said, a local system $V_\rho$ of $\CC$-vector spaces with a fixed framing at $x$, but also a local system of Lie algebras
\begin{equation}
\ad_\rho\subset\End(V_\rho)
\end{equation}
corresponding to the adjoint action of $\rho$ on the Lie algebra $\mathfrak{g}$ of $G$. Thus $\ad_\rho$ acts on $V_\rho$ and on $V_\rho\otimes W$, and its fiber at $x$ is canonically identified with $\mathfrak{g}$.

Let $L$ be the algebra of $\mathcal{C}^\infty$ differential forms on $X$ with values in $\ad_\rho$ and let $M$ be the algebra of $\mathcal{C}^\infty$ differential forms with values in $V_\rho \otimes W$. Then $L$ has a natural structure of DG Lie algebra, with the bracket combining the wedge product of forms with the Lie bracket of $\ad_\rho$ and with the differential coming from the flat structure of $\ad_\rho$; and $M$ has a structure of module over $M$ combining the wedge product of forms with the action of $\ad_\rho$ on $V_\rho$.

Furthermore, $L$ is equipped with an augmentation $\epsilon$ to $\mathfrak{g}$ which evaluates forms of degree zero at $x$.

\begin{theorem}[{Goldman-Millson~\cite{GoldmanMillson}, Budur-Wang~\cite[\S~7]{BudurWang}}]
\label{theorem:GoldmanMillsonBudurWang}
There are canonical isomorphisms of deformation functors
\begin{equation}
\Def(\rho)=\Def(L,\epsilon)
\end{equation}
and of sub-functors
\begin{equation}
\Def_k^i(\rho,W) = \Def_k^i(L,M,\epsilon) .
\end{equation}
\end{theorem}

We are now in a situation where we can apply our pro-representability methods of \S~\ref{section:pro-representability}--\ref{section:cone-module}: here $\epsilon$ is injective on $H^0(L)$ so the $L_\infty$ cone $C$ has $H^0(C)=0$, and $L$ and $M$ of course satisfy the appropriate boundedness and finite-dimensional hypotheses. So the above two deformation functors defined with DG Lie algebras are pro-representable and can be computed with $C$ instead of $L,\epsilon$.

\begin{corollary}
The local ring $\Ohat_\rho$ can be computed as the algebra $R$ that pro-represents $\Def(C)$, and the ideals $J_k^i(W)\subset \Ohat_\rho$ correspond to the jump ideals $J_k^i\subset R$ that pro-represent $\Def_k^i(C,M)$.
\end{corollary}

Here we can also use any other pair that is quasi-isomorphic (above $\mathfrak{g}$) to $(L,M)$.

\subsection{With mixed Hodge structure}
\label{section:with-MHS}
Now we turn exactly to the setting of \cite{Lefevre3}: $X$ is either a smooth complex algebraic varieties, and we work with a compactification $X\subset \overline{X}$ by a normal crossing divisor, or $X$ is the complement of such a divisor inside a compact Kähler manifold $\overline{X}$ (i.e.\ $X$ is quasi-Kähler).

We make the assumption that the local system $V_\rho$ underlies an \emph{admissible variation of mixed Hodge structure} (VMHS).  Such local systems appear naturally from geometry as $R^n f_* \underline{\QQ}_X$ for a morphism of varieties $f:Y\rightarrow X$: each fiber carries a MHS over a vector space locally identified with $H^n(f^{-1}(x), \QQ)$, the weight filtration varies so as to define sub-local systems over $\QQ$ and the Hodge filtration varies so as to define sub-holomorphic vector bundles. So abstractly here $V_\rho$ is a local system defined over the field $\kk$, with a filtration by sub-local systems $W_\bullet \subset V_\rho$ and a filtration by holomorphic sub-vector bundles $\mathcal{F}^\bullet \subset V_\rho \otimes \mathcal{O}_X$, such that at each $x\in X$ the fiber $(V_{\rho,x}, W_{\bullet,x}, \mathcal{F}^\bullet_x)$ forms a mixed Hodge structure; furthermore the induced flat connection $\nabla$ satisfies $\nabla(\mathcal{F}^p) \subset \mathcal{F}^{p-1}$. The admissibility condition also requires graded-polarizability and quasi-unipotent monodromy at infinity around $D$.

In this situation, the cohomology groups $H^n(X, V_\rho)$ carry a natural MHS, which is computed from an appropriate MHC. The Hodge filtration is the natural combination of the one of $X$ and the one of $V_\rho$. The weight filtration is more difficult to describe unless $X$ is already compact: for a variation of \emph{pure} Hodge structure of weight $k$ on $X$, the weight of $H^n(X, V_\rho)$ is $n+k$. The local system $\ad_\rho\subset \End(V_\rho)$ also underlies a VMHS and so $H^n(X, \ad_\rho)$ carry MHS, and the induced Lie bracket and action maps are morphisms of MHS. More generally we can fix another such VMHS $W$ on $X$, then $\ad(\rho)$ acts on $V_\rho\otimes W$ in the category of admissible VMHS and $H^\bullet(X, \ad_\rho)$ acts on $H^\bullet(X, V_\rho \otimes W)$ via morphisms of MHS.

As we intensively motivated in our previous work we want to compute the MHS on $H^n(X,\ad(\rho))$ via a MHC $L$ that also has a structure of DG Lie algebra \emph{at the level of chain complex}. Thus it contains information on both the MHS on cohomology and on the deformation functor $\Def(\rho)$. This is not so trivial: in the naive construction of MHC the multiplicative structure over $\kk$ is not commutative at the cochain level but only up to homotopy, and the ideas to solve this come from rational homotopy theory; we use polynomial rational differential forms, easy to handle with the Thom-Whitney functors of~\cite{Navarro}.

Similarly we want now a pair $(L,M)$ of MHC that is also a DG Lie pair and computes
\begin{equation}
(H(X,\ad_\rho), H(X,V_\rho\otimes W)),
\end{equation}
what we call a MHC + Lie pair. Here we can use directly our previous work which is formulated in terms of \emph{lax symmetric monoidal functors}. Briefly, this means that in~\cite{Lefevre3} we defined a functor $\MHC(X,\overline{X},\mathbb{V})$ from VMHS with unipotent monodromy at infinity to MHC, depending on the compactification, such that for two such VMHS $\mathbb{V}_1, \mathbb{V}_2$ there is a natural map
\begin{equation}
\MHC(X,\overline{X}, \mathbb{V}_1) \otimes \MHC(X,\overline{X}, \mathbb{V}_2) \longrightarrow \MHC(X,\overline{X}, \mathbb{V}_1 \otimes \mathbb{V}_2)
\end{equation}
that commutes (strictly, and only not up to homotopy) with the interchange map of the two factors of the tensor products. So this framework gives us now the Lie module:

\begin{theorem}
Let $X$ be either algebraic smooth, or quasi-Kähler, let $V_\rho$ and $W$ be two admissible VMHS over $X$. Then $\Ohat_\rho$ has a canonical MHS and the jump ideals $J_k^i(W)\subset \Ohat_\rho$ are sub-MHS.
\end{theorem}

\begin{proof}
This only requires taking into account the Lie module into the proof of \cite[Theorem~7.4]{Lefevre3}.

First fix the compactification $X\subset \overline{X}$ and assume that both $V_\rho$ and $W$ have unipotent monodromy at infinity.

Then $\ad_\rho$ and $V_\rho\otimes W$ are again such VMHS.
So we form
\begin{align}
L &:= \MHC(X,\overline{X}, \ad_\rho)\\
M &:=\MHC(X, \overline{X}, V_\rho \otimes W)
\end{align}
using the lax symmetric monoidal $\MHC$ functor constructed through \cite[\S4--6]{Lefevre3}. So $L$ is also naturally a DG Lie algebra with bracket the composition
\begin{equation}
\MHC(X,\overline{X}, \ad_\rho) \otimes \MHC(X,\overline{X}, \ad_\rho) \longrightarrow \MHC(X,\overline{X}, \ad_\rho\otimes \ad_\rho) \stackrel{\text{bracket}}{\longrightarrow} \MHC(X,\overline{X}, \ad_\rho)
\end{equation}
and $M$ is a module over $L$ with action map
\begin{equation}
\MHC(X,\overline{X}, \ad_\rho) \otimes \MHC(X,\overline{X}, V_\rho \otimes W) \longrightarrow \MHC(X,\overline{X}, {\ad_\rho} \otimes {V_\rho} \otimes W) \stackrel{\text{action}}{\longrightarrow} \MHC(X,\overline{X}, V_\rho \otimes W) .
\end{equation}
There is also an augmentation
\begin{equation}
\epsilon: L \longrightarrow \mathfrak{g}
\end{equation}
which is a morphism of MHC and of DG Lie algebras. Two different compactifications of $X$ lead to quasi-isomorphic MHC. If the monodromy is only quasi-unipotent one can work equivariantly over a finite cover of $X$ on which the monodromy of both $V_\rho$ and $W$ is unipotent.

From here there is everything to apply our theory: $(L,M,\epsilon)$ is an MHC + augmented Lie pair satisfying all the hypotheses of Theorem~\ref{Theorem2-section} in \S~\ref{section:proof-main}, so there is a MHS on the algebra $R$ that pro-represents $\Def(L,\epsilon)$ and a sub-MHS on the ideals $J_k^i \subset R$ that pro-represent $\Def_k^i(L,M,\epsilon)$. By the previous Theorem~\ref{theorem:GoldmanMillsonBudurWang} and its corollary, this puts a MHS on $\Ohat_k^i$ and shows that the $J_k^i(W) \subset \Ohat_\rho$ are sub-MHS.
\end{proof}

\section{Proof of the main theorem}
\label{section:proof-main}
We come to the core of this article and prove our Theorem~\ref{Theorem2-section} below in a sequence of lemmas, starting from the data that we get just in the previous section from the geometric situation.

So we assume that we are given a MHC + augmented Lie pair $(L,M,\epsilon)$. In particular:
\begin{itemize}
\item $L$ is a DG Lie algebra and $M$ is a DG module over $L$,
\item $\epsilon$ is an augmentation of $L$ to a Lie algebra $\mathfrak{g}$,
\item $L, M$ are both MHC over the base field $\kk\subset\RR$,
\item $\mathfrak{g}$ carries a MHS, and we consider it as a MHC concentrated in degree zero.
\end{itemize}
All of these structures are mutually compatible.

Furthermore, through the whole section we make the following hypotheses, which are clearly satisfied when coming from geometry:
\begin{itemize}
\item $L$ and $M$ have finite-dimensional cohomology in each degree,
\item $H^{<0}(L)=0$,
\item $M$ is bounded above,
\item the restriction of $\epsilon$ to $H^0(L)$ is injective.
\end{itemize}

In particular, when one forms the cone $C$ of $\epsilon$ then $H^0(C)=0$. So there is a pro-Artin algebra $R$ over $\kk$ that pro-represents $\Def(L,\epsilon)$ and there are jump ideals $J_k^i(M\otimes R, d_u)\subset R$ corresponding to $\Def_k^i(L,M,\epsilon)$. Our goal is to show:

\begin{theorem}
\label{Theorem2-section}
In this situation, $R$ has a MHS (as projective limit of MHS on $R_n:=R/\mathfrak{m}_R^n$) and the jump ideals are sub-MHS (in each $R_n$). This MHS is invariant under an augmented quasi-isomorphism of $(L,M)$.
\end{theorem}

First we want to form the $L_\infty$ cone $C$ of $\epsilon$. 

\begin{lemma}
The pair $(C,M)$ has a structure of absolute Hodge complex + $L_\infty$ pair.
\end{lemma}

\begin{proof}
First, applying the functor $\Dec$ to the whole MHC + augmented Lie pair $(L,M,\epsilon)$ gives us an AHC + Lie pair: on the one hand it sends MHC to AHC and on the other hand it is lax symmetric monoidal (using the map~\eqref{equation:dec-symmetric-monoidal}) so preserves this compatibility with the algebraic structures.

Then form the cone as explained in \S~\ref{section:AHC}, which is again an AHC. And it is straightforward to see (with the formulas in \S~\ref{section:the-cone}) that the $L_\infty$ structure of $C$ is compatible with its AHC structure, and that the action of $C$ on $M$ is also compatible with these AHC. Indeed in the category of AHC there is no problem of shift of the weight filtration so these multilinear operations preserve the weight filtration in the simplest way.
\end{proof}

Then we apply a new powerful theorem, the homotopy transfer of structure in mixed Hodge theory.

\begin{lemma}
\label{lemma:apply-transfer}
The pair $(H(C),H(M))$ has an induced $L_\infty$ structure compatible with its MHS and $(C,M)$ is quasi-isomorphic, in the sense of AHC + $L_\infty$ pair, to it.
\end{lemma}

\begin{proof}
This follows essentially from \cite[\S~5]{CiriciSopena} (see also Remark~5.7 there) which just appeared. However the theory has been studied in much greater detail for the case of MHC + commutative algebras, for the needs of rational homotopy theory and for putting MHS on the rational homotopy groups, so let us explain more. There are two essential steps.

The first essential step of the classical story is: starting from an MHC + commutative algebra $A$ one can build a \emph{minimal model} for $A$, that is, a commutative DG algebra $\mathscr{A}$ with a quasi-isomorphism to $A$, built inductively by adding free generators so as to induce an isomorphism in cohomology in higher and higher degrees, and to put a MHS (infinite-dimensional) on $\mathscr{A}$ that on cohomology induces the MHS on $H(A)$. This was first done by Morgan \cite{Morgan}. One difficulty is that this whole construction is only unique up to homotopy and this requires developing the appropriate algebra to deal with it.

The hypothesis $H^0(A)=0$ is crucial here for the inductive step, else the newly added free generators to $\mathscr{A}$ of degree $n$ get combined with elements of degree $0$, producing new elements of degree $n$ and the induction cannot continue.

Then this theory has been well re-written with tools of homotopical algebra in the work of J.\  Cirici and collaborators. It is in \cite{CiriciGuillenComplexes} without multiplicative structures and in \cite{CiriciGuillen} and \cite{Cirici} for commutative algebras.  This works equally well with AHC because at each step are added to the minimal model pieces of the MHS of $H(A)$ : the essential lemma is \cite[Lemma~3.12]{CiriciGuillen}. The theory of minimal models, without MHS but for many more algebras over operads, is written in \cite{CiriciRoig}; this one works for $L_\infty$ algebras as stated there. Combining both gives minimal models for AHC + $L_\infty$ algebras.

To deal with the module we argue as follows. Starting from $(C,M)$ we build a minimal model $\mathscr{C}\rightarrow C$ of $L_\infty$ algebras and a minimal model $\mathscr{M}\rightarrow M$ as cochain complex (here this doesn't require any condition on $H^0(M)$ : this is Beilinson's equivalence between AHC and the derived category of MHS, re-written in \cite[\S~4]{CiriciGuillenComplexes}). Then we see the action map as a morphism of $L_\infty$ algebras $C\rightarrow \End(M)$ and we use the cofibration property of the minimal model (this is the main topic of \cite{Cirici}) to lift it to a map $\mathscr{C}\rightarrow\End{\mathscr{M}}$.

The second essential step, again written for commutative DG algebras, is: starting from the minimal model $\mathscr{A}$ with its MHS, develop the homotopy transfer theorem from this $\mathscr{A}$ to $H(A)$. This is the content of \cite{CiriciSopena} and it applies as well to $L_\infty$ algebras.

Here, to go from $L_\infty$ algebras to $L_\infty$ pairs, we form the $L_\infty$ algebra $\mathscr{C}\oplus\mathscr{M}$ as in~\eqref{equation:L-infinity-pair-sum} (this is actually how Budur-Rubi{\'o} prove their transfer of structure for $L_\infty$ pairs). We see again directly that this is a $L_\infty$ algebra with a compatible MHS and hence the homotopy transfer applies to this, from which we recover separately $H(\mathscr{C})=H(C)$ and $H(\mathscr{M})=H(M)$, with the MHS and all operations.
\end{proof}

From now on there are no more AHC or MHC; only complexes of MHS. So we have quite a lot of structure but we are only doing linear algebra. Our goal is to develop the analogue of \S~\ref{section:pro-representability} with MHS everywhere.

\begin{lemma}
The algebra $R$ that pro-represents $\Def(C)$ has a canonical MHS.
\end{lemma}

\begin{proof}
This is what is done in \cite[~9]{Lefevre2} but we are now able to give a simpler proof, avoiding the somewhat unnatural hypothesis of Theorem~8.4 there (working with the cone as MHC and not as AHC).

Namely, seeing the definition of the functor $\mathscr{C}$ we see directly that $\mathscr{C}(H(C))$ is a coalgebra which is an inductive limit of finite-dimensional ones carrying a MHS. Denote these by $\mathscr{C}_n(H(C))$. Then the inductive limit of the finite-dimensional coalgebras $H^0(\mathscr{C}_n(H(C)))$ is dual to the projective limit of the Artin algebras $R_n$. This puts a MHS on $R$ as pro-Artin algebra.

Lastly we need to show that this is independent of the choices made in the homotopy transfer of structure. But in \cite{CiriciSopena} the transferred operations are related by a homotopy (in this algebraic, infinity sense) that are themselves again compatible by the MHS. So we only have to apply the classical proof that taking this $H^0(\mathscr{C}(H(C)))$ is invariant under homotopy of the operations of $H(C)$, in the category of MHS. See \cite{LodayVallette} for full details.
\end{proof}

Then we form the universal complex $H(M)\otimes R$, that we consider as the projective limit of the complexes $H(M)\otimes R_n$. It receives in each degree a MHS, by tensor product of the MHS of $H(M)$ and the one of $R_n$.

\begin{lemma}
The universal derivation $d_u:=d_{\omega_u}$ associated to the universal Maurer-Cartan element
\begin{equation*}
\omega_u\in \MC(H(C))(R_n)
\end{equation*}
on the universal complex $H(M)\otimes R_n$ is a morphism of MHS.
\end{lemma}

\begin{proof}
The Maurer-Cartan equation defines a functor
\begin{equation}
\MC^{\mathbf{MH}}(H(C)):\mathbf{Art+MHS}\longrightarrow\mathbf{Set}
\end{equation}
by taking the Maurer-Cartan elements in $H(C)\otimes \mathfrak{m}_A$ that are defined over $\kk$ and in $W^0 \cap F^0$. This is the same thing as, for $A \in \mathbf{Art+MHS}$ with $\mathfrak{m}_A^n=0$,
\begin{equation}
\MC^{\mathbf{MH}}(H(C))(A) = \Hom_{\mathbf{CoAlg+MHS}}(A^*, H^0(\mathscr{C}_n(H(C)))) .
\end{equation}
The universal Maurer-Cartan element $\omega_u$ at order $n$ corresponds to the image in $\MC^{\mathbf{MH}}(H(C))(R_n)$ of the identity $\id \in \Hom(R^{**}_n, R_n)$ (here everything is finite-dimensional so $R_n^{**}=R_n$). So we see that, because of its universal property, $\omega_u$ is defined over $\kk$ and in $W_0 \cap F^0$.

Finally, we see then from formula~\eqref{equation:derivation-M} that the corresponding universal derivation respects the MHS.
\end{proof}

\begin{remark}
Here we could have written our results in slightly more generality: we actually have shown that for any $A\in \mathbf{Art+MHS}$ and for \emph{any} $\omega\in \MC^{\mathbf{MH}}(H(C))(A)$ the derivation $d_\omega$ on $H(M)\otimes A$ preserves the MHS; we call these $\omega$ \emph{mixed Hodge Maurer-Cartan elements}. From the Tannakian point of view \S~\ref{section:tannakian}, such an $\omega$ is a fixed point of the action of the Tannakian group $\mathscr{G}$ and this implies that the differential that it defines is $\mathscr{G}$-equivariant, that is, a morphism of MHS; see the argument below.
\end{remark}

To sum up, what we get on $H(M)\otimes R_n$ is a whole complex of free $R_r$-modules with a MHS on each term and a derivation $d_u$ which is in each degree a morphism of MHS; this also goes to the projective limit.

\begin{lemma}
The jump ideals $J_k^i(H(M)\otimes R, d_u) \subset R$ are sub-MHS.
\end{lemma}

\begin{proof}
Work over $R_n$ and form the map
\begin{equation}
d_u^{i-1}\oplus d_u^i : (H^{i-1}(M)\otimes R_n)\oplus (H^i(M)\otimes R_n) \longrightarrow (H^i(M)\otimes R_n) \oplus (H^{i+1}(M) \otimes R_n)
\end{equation}
Denote by $P$ the left-hand side, $Q$ the right-hand side and $D$ this map. By choosing bases over $\kk$ of $H^\bullet(M)$, $P$ and $Q$ are free finitely generated $R_n$-modules.

Let us first deal with the jump ideals $J_k^i(H(M)\otimes R_n, d_u)\subset R_n$ for $n_i-k+1=1$ ($n_i=\dim(H^i(M))$), that is, the ideals generated by the coefficients of a matrix of $D$.

Using the Tannakian point of view of \S~\ref{section:tannakian} we can pretend that we are working with $\kk$-vector spaces with a $\kk$-linear action of the group $\mathscr{G}$:
\begin{itemize}
\item $R_n$ carries a $\mathscr{G}$-action as algebra,
\item $P,Q$ carry a $\mathscr{G}$-action as $R_n$-modules,
\item $D$ is $R_n$-linear and $\mathscr{G}$-equivariant,
\item we want to show that the ideal of $R_n$ generated by the coefficients of a matrix of $D$ (which is independent of the bases) is stable under the $\mathscr{G}$-action.
\end{itemize}
So let $(e_j)$ be a basis of $P$, $(f_i)$ be a basis of $Q$, over $R_n$. This defines a matrix $(a_{ij})$ of $D$, such that $D(e_j)=\sum_i a_{ij} f_i$. Let $I\subset R_n$ be the ideal generated by the $(a_{ij})$. Let $g\in\mathscr{G}$. Then $g\cdot I$ is the ideal generated by the $(g\cdot a_{ij})$. But this is the matrix of $D$ in the bases $(g\cdot e_j)$ and $(g\cdot f_i)$, indeed
\begin{equation}
D(g\cdot e_j) = g\cdot D(e_j) = g\cdot \big(\sum_i a_{ij} f_i \big) = 
\sum_i (g\cdot a_{ij}) (g\cdot f_i)
\end{equation}
using the $\mathscr{G}$-equivariance. Since $I$ is independent of the bases this gives $g\cdot I = I$.

The argument is the same for the general case of the jump ideals generated by minors of size $n_i-k+1$ of $D$: with the same notations but letting $I$ be the ideal generated by the minors of $(a_{ij})$ then $g\cdot I$ is the ideal generated by the minors of $(g\cdot a_{ij})$ so it is again $I$.
\end{proof}

\begin{remark}
In \cite{Lerer} a related result is proven working with gradings instead of MHS. In the case where the variety $X$ is compact and $\rho$ is the monodromy of a \emph{complex} VHS, then the above $L$ and $M$ carry on their cohomology pure \emph{complex} MHS and furthermore the pair $(L,M)$ is formal. From this it is possible to construct, in a naive way, a \emph{split} $\CC$-MHS on $\Ohat_\rho$ (first introduced in \cite[\S~2.3]{EyssidieuxSimpson}); this doesn't require to work with the cone because splitting the short exact sequence
\begin{equation}
0 \longrightarrow \mathfrak{g}/\epsilon(H^0(L)) \longrightarrow H^1(C) \longrightarrow H^1(L) \longrightarrow 0
\end{equation}
is equivalent to splitting the weight filtration (the left-hand side has weight $0$ and the right-hand side has weight $1$). So we can essentially work with a bigrading on the pure $\CC$-MHS of $H(L)$ and $H(M)$. The result of \cite[\S~4]{Lerer} is that the jump ideals are sub-$\CC$-MHS.

Of course, we recover this result by splitting our final MHS. But we would like to point out that
if we worked in some way with gradings instead of MHS (see the next section), then
the homotopy transfer of structure holds and our method produces gradings on $R$. And to use the previous lemma it suffices to replace the Tannakian group $\mathscr{G}$ by $\mathbb{G}_m$ (or $\mathbb{G}_m^2$ for bigradings), whose category of representations is equivalent to graded objects.
\end{remark}

\section{Application to the global structure}
\label{section:application}

A powerful method introduced by Budur-Rubi\'o is to use the weight grading on $H(L)$ and $H(M)$ and its compatibility with the higher operations. They manage to show that under the assumption that $H^1(L)$ has no weight zero or below, then there are only finitely many non-vanishing higher action maps. This is roughly because the action of elements of $H^1(L)$ on $H(M)$ of weight greater than one must increase the weight, but $H(M)$ has only finitely many weights. Then for $G=\CC^*$ this property forces the global loci $\Sigma_k^i$ to be sub-tori. In \cite{Lefevre3} we used similar ideas and we show that under this same hypothesis on $H^1(L)$ then there are only finitely many non-vanishing higher operations needed to write the deformation functor. A consequence of this is that $\Ohat_\rho$ has a presentation by generators and a finite number of polynomial relations, and the degree of those is related to the weights of the MHS on $H^2(L)$.

In those two earlier articles, the assertion that the higher operations are compatible with the weight comes from a result of Cirici-Horel \cite{CiriciHorel}. But now, we recover this directly with the transfer of structure with MHS, and we would like to get some consequences of the present theory.

So we work with the group $G=\CC^*$. The moduli space of representations of $\pi_1(X.x)$ is simply denoted by $\Mb(X,1)$ and it is also the moduli space of local systems, since the action of $G$ by conjugation is trivial and all local systems are stable. Then $\Mb(X,1)$ is isomorphic to the product of a complex torus $(\CC^*)^{b_1(X)}$ and a finite group. Here is what we get.

\begin{theorem}
Let $X$ be smooth algebraic or quasi-Kähler, let $W$ be an admissible VMHS over $X$. The irreducible components of the relative jump loci $\Sigma_k^i(W)\subset \Mb(X,1)$ at a representation $\rho$ are translated sub-tori.
\end{theorem}

\begin{proof}
In rank one, the local system $\ad_\rho$ is trivial and it acts trivially on $V_\rho \otimes W$. The cohomology pair is $(H(X,\CC), H(X,V_\rho \otimes W))$. Let $(L,M)$ be the Lie pair computing this. Then $L$ is abelian, but the action of $L$ on $M$ is not trivial (it contains the multiplication of differential forms). So on the cohomology pair with $L_\infty$ structure, $H(L)$ is abelian but there are higher action maps on $H(M)$.

Doing this with the cone $C$, then $H^0(C)=0$ and $H^n(C)=H^n(L)$ for $n\geq 1$. So there are also no higher operations on $H(C)$, but higher action maps of $H(C)$ on $H(M)$.

Consider $V_\rho$ as a VMHS of weight zero in the most trivial way. Because of the results of the previous section and in particular Lemma~\ref{lemma:apply-transfer}, the higher operations can be taken to be strictly compatible with the weight filtrations that exist on $H(L)=H(X,\CC)$ and on $H(M)=H(X,V_\rho \otimes W)$. Take a Maurer-Cartan element $\omega$ of $H(L)$: this is just an element of the vector space $H^1(X,\CC)$ (with no equations), then it has weight greater than $1$ for \emph{smooth} varieties. This implies that the maps
\begin{equation}
\label{equation:operations-increase-weight}
m_n(\omega,\dots,\omega,-) : H^\bullet(M) \longrightarrow H^{\bullet+1}(M)
\end{equation}
increase the weights by at least $n-1$ ($m_n$ has $n-1$ arguments $\omega$ of weight $\geq 1$). Since $H(M)$ is bounded above, and has in each degree finitely many weights, then only finitely many of the above maps can be non-zero.

This is the setting for applying \cite[Theorem~6.1]{BudurRubio}. Roughly, the sub-torus property is implied by the fact that there is an exponential map $H^1(X, \CC)\rightarrow \Mb(X,1)$, carrying the algebraic set called \emph{resonance variety}
\begin{equation}
\mathcal{R}_k^i := \{ \omega \in H^1(L)\ |\ \dim H^i(H(M),d_\omega) \geq k \}
\end{equation}
into the algebraic set $\Sigma_k^i(W)$ and inducing an isomorphism of germs
\begin{equation}
R/J_k^i\simeq \Ohat_\rho/J_k^i(W)
\end{equation}
at $0$ and $\rho$. This an exponential Ax-Lindemann Theorem as explained in \cite{BudurWangQuasi}. The condition that there are only finitely many action maps is needed in order to show that the $\mathcal{R}_k^i$ are indeed closed sub-schemes of $H^1(X,\CC)$.
\end{proof}

For $\rho$ the trivial representation and $W$ the trivial VMHS of rank one, this is the argument of Budur-Rubi\'o.

\begin{remark}
If we could develop our theory also for singular varieties, then the condition $W_0 H^1(X,\CC)=0$ would be necessary: this is the condition under which the maps \eqref{equation:operations-increase-weight} strictly increase the weights and then under which one can guarantee that there are only finitely many of those that do not vanish. This same condition appears also in the work of Esnault-Kerz and in our \cite{Lefevre3}.
\end{remark}

This conclusion is not new when $X$ is algebraic and follows from the theory developed in \cite{BudurWangAbsolute}, where even 
more general relative jump loci are considered. However this theory is heavily algebraic and relies among other on the comparison with the moduli space of vector bundles with an integrable connection. Thus we generalize this here to quasi-Kähler manifolds and we give a Hodge-theoretic proof more in the spirit of our previous work. Also, there is discussed a particular notion of local systems (or more generally, sheaves or complexes of sheaves) of \emph{geometric origin}; here we somehow replace this with an abstract VMHS.

\bibliographystyle{amsalpha}
\bibliography{Bibliographie}

\end{document}